\documentclass[oneside,british,english]{amsart}
\usepackage[T1]{fontenc}
\usepackage[latin9]{inputenc}
\usepackage{babel}
\usepackage{amsthm}
\usepackage{amstext}
\usepackage{amssymb}
\usepackage{esint}
\usepackage[unicode=true,
 bookmarks=true,bookmarksnumbered=false,bookmarksopen=false,
 breaklinks=false,pdfborder={0 0 1},backref=false,colorlinks=false]
 {hyperref}
\hypersetup{
 pdfauthor={Giovanni S. Alberti}}

\makeatletter
\theoremstyle{plain}
\newtheorem{thm}{\protect\theoremname}
  \theoremstyle{definition}
  \newtheorem{defn}[thm]{\protect\definitionname}
  \theoremstyle{definition}
  \newtheorem{example}[thm]{\protect\examplename}
  \theoremstyle{plain}
  \newtheorem{lem}[thm]{\protect\lemmaname}
  \theoremstyle{plain}
  \newtheorem{prop}[thm]{\protect\propositionname}
  \theoremstyle{plain}
  \newtheorem{cor}[thm]{\protect\corollaryname}
  \theoremstyle{remark}
  \newtheorem{rem}[thm]{\protect\remarkname}

\usepackage{enumitem}
\setlist{leftmargin=*}
\usepackage[nocompress]{cite}

\makeatother

  \addto\captionsbritish{\renewcommand{\corollaryname}{Corollary}}
  \addto\captionsbritish{\renewcommand{\definitionname}{Definition}}
  \addto\captionsbritish{\renewcommand{\examplename}{Example}}
  \addto\captionsbritish{\renewcommand{\lemmaname}{Lemma}}
  \addto\captionsbritish{\renewcommand{\propositionname}{Proposition}}
  \addto\captionsbritish{\renewcommand{\remarkname}{Remark}}
  \addto\captionsbritish{\renewcommand{\theoremname}{Theorem}}
  \addto\captionsenglish{\renewcommand{\corollaryname}{Corollary}}
  \addto\captionsenglish{\renewcommand{\definitionname}{Definition}}
  \addto\captionsenglish{\renewcommand{\examplename}{Example}}
  \addto\captionsenglish{\renewcommand{\lemmaname}{Lemma}}
  \addto\captionsenglish{\renewcommand{\propositionname}{Proposition}}
  \addto\captionsenglish{\renewcommand{\remarkname}{Remark}}
  \addto\captionsenglish{\renewcommand{\theoremname}{Theorem}}
  \providecommand{\corollaryname}{Corollary}
  \providecommand{\definitionname}{Definition}
  \providecommand{\examplename}{Example}
  \providecommand{\lemmaname}{Lemma}
  \providecommand{\propositionname}{Proposition}
  \providecommand{\remarkname}{Remark}
\providecommand{\theoremname}{Theorem}

\begin{document}
\global\long\def\R{\mathbb{R}}

\global\long\def\N{\mathbb{N}}

\global\long\def\C{\mathbb{C}}

\global\long\def\Cl{C}

\global\long\def\tr{\mathrm{tr}}

\global\long\def\phi{\varphi}

\global\long\def\epsilon{\varepsilon}

\global\long\def\div{{\rm div}}

\global\long\def\Hcurl{H(\curl,\Omega)}

\global\long\def\Hmu{H^{\mu}(\curl,\Omega)}

\global\long\def\Hocurl{H_{0}(\curl,\Omega)}

\global\long\def\Hdiv{H(\div,\Omega)}

\global\long\def\k{\omega}

\global\long\def\Cone{C^{1}(\overline{\Omega};\C)}

\global\long\def\curl{{\rm curl}}

\global\long\def\supp{{\rm supp}}

\global\long\def\sp{{\rm span}}

\global\long\def\bo{\partial\Omega}

\global\long\def\eig{\phi}

\global\long\def\g{g}

\global\long\def\q{q}

\global\long\def\A{\mathcal{A}}

\global\long\def\e{\mathbf{e}}

\selectlanguage{british}%
\global\long\def\ldr{L^{2}(\Omega;\R)}

\global\long\def\Vr{H_{0}^{1}(\Omega;\R)}

\global\long\def\Coner{\Cl^{1}(\overline{\Omega};\R)}

\global\long\def\V{H_{0}^{1}(\Omega;\C)}

\global\long\def\mina{\lambda}

\global\long\def\maxa{\Lambda}

\global\long\def\Honer{H^{1}(\Omega;\R)}

\selectlanguage{english}%
\title[Critical Points of Solutions to the Helmholtz Equation in 3D]{Absence of Critical Points of Solutions to the Helmholtz Equation in 3D}
\author{Giovanni S.~Alberti}
\address{Department of Mathematics and Applications,  \'Ecole Normale Sup\'erieure, 45 rue d'Ulm, 75005 Paris, France.} 
\email{giovanni.alberti@ens.fr}
\thanks{This work was partially supported by  the EPSRC Science \& Innovation Award to the Oxford Centre for Nonlinear PDE (EP/EO35027/1) and by the ERC Advanced Grant Project MULTIMOD-267184.}
\date{22\textsuperscript{nd} January 2016}
\begin{abstract}
The focus of this paper is to show the absence of critical points
for the solutions to the Helmholtz equation in a bounded domain $\Omega\subset\R^{3}$,
given by
\[
\left\{ \begin{array}{l}
-\div(a\,\nabla u_{\omega}^{g})-\omega qu_{\omega}^{g}=0\quad\text{in \ensuremath{\Omega},}\\
u_{\omega}^{g}=g\quad\text{on \ensuremath{\partial\Omega}.}
\end{array}\right.
\]
We prove that for an \emph{admissible} $g$ there exists a finite
set of frequencies $K$ in a given interval and an open cover $\overline{\Omega}=\cup_{\omega\in K}\Omega_{\omega}$
such that $|\nabla u_{\omega}^{g}(x)|>0$ for every $\omega\in K$
and $x\in\Omega_{\omega}$. The set $K$ is explicitly constructed.
If the spectrum of this problem is simple, which is true for
a generic domain $\Omega$, the admissibility condition on $g$ is
a generic property.
\end{abstract}
\maketitle
\maketitle

\section{Introduction}

The Rad\'o-Kneser-Choquet theorem states that the solutions of the Dirichlet
boundary value problem for the Laplace equation in the unit disk $\mathbb{D}$
of $\R^{2}$ given by
\[
\left\{ \begin{array}{l}
-\Delta u^{i}=0\quad\text{in }\mathbb{D},\\
u^{i}=g_{i}\quad\text{on }\partial\mathbb{D},
\end{array}\right.
\qquad i=1,2,
\]
where $(g_{1},g_{2})\colon\partial\mathbb{D}\to\R^2$ is a homeomorphism of the unit sphere onto
the boundary of a bounded convex set, satisfy
\begin{equation}
|\det(\nabla u^{1},\nabla u^{2})(x)|>0,\qquad x\in\mathbb{D}.\label{eq:constraint rado}
\end{equation}
Formulated by Rad\'o in terms of harmonic functions \cite{RADO-1926},
this result was proved later by Kneser \cite{KNESER-1926} and Choquet
\cite{choquet-1945}. Both proofs are based on the absence of critical
points of harmonic functions in $\mathbb{D}$ having certain boundary
values. The reader is referred to \cite{DUREN-2004} for a review
of these topics. The relevance of this result (and of its generalizations
discussed below) goes beyond the theory of harmonic mappings, since
the constraint given in \eqref{eq:constraint rado} appears in many
contexts, such as in elasticity theory, homogenization \cite{ALESSANDRINI-NESI-01}
and inverse problems \cite{bal2012_review}.

This result has been extended to the conductivity equation in a bounded
convex set $\Omega\subset\R^{2}$
\[
\left\{ \begin{array}{l}
-\div(a\nabla u^{i})=0\quad\text{in }\Omega,\\
u^{i}=g_{i}\quad\text{on }\partial\Omega,
\end{array}\right.
\qquad i=1,2,
\]
where $a\in C^{0,\alpha}(\overline{\Omega};\R^{2\times 2})$ and is
uniformly elliptic. Under the same assumptions on the boundary values,
there holds
\[
|\det(\nabla u^{1},\nabla u^{2})(x)|>0,\qquad x\in\Omega,
\]
see \cite{ALESSANDRINI-NESI-01,bauman2001univalent,alessandrini-nesi-2015,alberti-capdeboscq-2015}.
As above, the proof is based on the fact that if $\det(\nabla u^{1},\nabla u^{2})(x)=0$ for some $x\in\Omega$,
then a linear combination $u$ of $u^{1}$ and $u^{2}$ has a critical
point in $x$, i.e.\ $\nabla u(x)=0$, and $-\div(a\nabla u)=0$ in $\Omega$ by linearity. This simple argument shows that the essence of the Rad\'o-Kneser-Choquet theorem lies in the absence of critical points. In other words, the study of the Jacobian for systems of equations reduces to the study of the critical set for the scalar conductivity equation. The absence of critical points of solutions with suitable boundary conditions is a consequence
of  the topological properties of the two-dimensional space and of the maximum principle
\cite{alessandrini1986,alessandrinimagnanini1994}.

In three dimensions, the Rad\'o-Kneser-Choquet theorem and its generalizations to non-constant
conductivities $a$ completely fail. Laugesen
\cite{laugesen-1996} proved that there exists a self-homeomorphism
of $S^{2}$ such that the Jacobian of its harmonic extension vanishes
at some point (see also \cite[Section 3.7]{DUREN-2004}). Similar
negative results for the case $a\not\equiv 1$ have been
obtained by Briane et al. \cite{briane-milton-nesi-2004} and by Capdeboscq
\cite{cap-2015}. This phenomenon is clearly connected with the fact
that in three dimensions the set of critical points is one-dimensional,
while in 2D it consists of isolated points \cite{han-1994,hardt-1999,han-lin}. Similarly, the use of the maximum principle is essential: for instance, the solutions to $-\Delta u - u  = 0$ are oscillatory, and in general critical points will occur.

Motivated by these negative results, in this paper we study the absence of critical points in a situation where none of these two requirements are fulfilled:  the three-dimensional Helmholtz equation
\begin{equation}
\left\{ \begin{array}{l}
-\div(a\,\nabla u_{\k}^{g})-\k\q u_{\k}^{\g}=0\qquad\text{in \ensuremath{\Omega},}\\
u_{\k}^{\g}=\g\qquad\text{on \ensuremath{\partial\Omega},}
\end{array}\right.\label{eq:helmholtz-intro}
\end{equation}
where $\q\in L^{\infty}(\Omega;\R_{+})$
and $a\in C^{0,\alpha}(\overline{\Omega};\R^{3\times3})$ and is uniformly
elliptic.  As explained for the conductivity equation, the study of the critical set is closely related to the study of the Jacobian, which would correspond to the Rad\'o-Kneser-Choquet theorem for the Helmholtz equation in 3D. However, such extension does not seem immediate
at this stage (see Section~\ref{sec:Towards-the-study} for an intermediate result), and remains
an interesting open problem. In general, it is not possible to show the absence of critical points in the whole domain for a fixed choice of the boundary
value and of the frequency, because of the obstructions discussed above: at best, only the local absence of critical points can be proved. In this work, this is achieved  by using multiple frequencies $\omega$ in a fixed range
$\A=[K_{min},K_{max}]$, 
which allows to compensate the
lack of the maximum principle and the dimensionality issue.

The main result of this work states that if $g$ is \emph{admissible} then
there exists a finite set of frequencies $K\subseteq\A$ and a finite open
cover $\overline{\Omega}=\cup_{\omega\in K}\Omega_{\omega}$ such
that $|\nabla u_{\omega}^{\g}(x)|>0$ for every $\omega\in K$ and
$x\in\Omega_{\omega}$ (Theorem~\ref{thm:genericity}). The set $K$
is explicitly constructed. This result shows that, if $\g$ is admissible, the
critical set $\{x\in\overline{\Omega}:|\nabla u_{\omega}^{\g}(x)|=0\}$ moves when
$\omega$ changes, and that with a finite number of frequencies it
is possible to enforce this constraint everywhere. The proof is based
on the holomorphicity of the map $\omega\mapsto u_{\omega}^{\g}$
and on the spectral analysis of \eqref{eq:helmholtz-intro}.  As such, the method considered in this paper is not well suited
for the fixed-frequency case.

The two-dimensional results discussed above required the boundary
values to be suitably chosen, in a way that was independent of the
conductivity. This seems to be unachievable in 3D if $a$ is not constant.
Our result is valid for \emph{admissible} boundary values $g$,
a concept that will be precisely defined in Definition~\ref{def:admissible}.
It is worth mentioning that in most cases, e.g. if the spectrum of
\eqref{eq:helmholtz-intro} is simple or if $\Omega$ is a cuboid
with $a\equiv q\equiv 1$, the admissibility is a generic property (Proposition~\ref{prop:admissible is generic}).
In particular, the admissibility is a generic property for a generic
bounded domain $\Omega$ in the class of $C^{2}$ domains, provided
that $a$ and $\q$ are $C^{2}$. Thus, our result will be valid for
a generic boundary condition in a generic bounded domain (Corollary~\ref{cor:genericity}).

This paper represents a generalization of several works \cite{alberti2013multiple,albertigsII,2014-albertigs,ammari2014admittivity,alberti-ammari-ruan-2014,alberti-capdeboscq-2014,alberti-capdeboscq-analytic},
in which multiple frequencies are used to enforce several non-zero local
constraints for certain frequency-dependent PDE. All these results
are based on the fact that the constraint under investigation is satisfied
for a particular choice of the frequency, e.g. in $\omega=0$. However,
as mentioned above, this cannot be done for the conductivity equation
in 3D if $a\not\equiv1$. The purpose of this work is to show that
this approach can be used even without this requirement.

In addition to its theoretical interest, the problem studied in this
paper, and similar problems related to different PDE or different
constraints, are motivated by the mathematical theory of hybrid imaging
inverse problems \cite{bal2012_review,kuchment-2012}. These are inverse problems with internal data, and such
constraints for the Helmholtz equation
give uniqueness and Lipschitz stability of the reconstruction. For instance, the Jacobian constraint naturally appears in microwave imaging by elastic deformation
\cite{cap2011,alberti2013multiple}, and the absence of critical points is needed in transient elastography \cite{bal2012_review}. It is worth mentioning that complex geometric
optics solutions or the Runge approximation property can be used in
3D to construct solutions to the conductivity equation or to the Helmholtz
equation with non-zero Jacobian, at least locally \cite{bal2011reconstruction,bal2012inversediffusion}.
The construction of suitable boundary conditions usually depends on
the unknown coefficients of the PDE. The approach of this paper is
different in nature, since almost any choice of the boundary value
suffices.

This paper is structured as follows. In Section~\ref{sec:Main-result}
we discuss the main results of this work. Section~\ref{sec:Proof-of-Theorem}
is devoted to the proof of Theorem~\ref{thm:genericity} and Section~\ref{sec:The-genericity-of}
to the genericity of the admissibility property. Finally, in Section~\ref{sec:Towards-the-study}
we discuss the issue of extending this work to the Jacobian for systems of equations and present
an intermediate result.

\section{\label{sec:Main-result}Main Results}

Let $\Omega\subseteq\R^{3}$ be a $C^{1,\alpha}$ bounded domain for
some $\alpha\in(0,1)$. We consider the Dirichlet boundary value problem
for the Helmholtz equation 
\begin{equation}
\left\{ \begin{array}{l}
-\div(a\,\nabla u_{\k}^{\g})-\k\q u_{\k}^{\g}=0\quad\text{in \ensuremath{\Omega},}\\
u_{\k}^{\g}=\g\quad\text{on \ensuremath{\partial\Omega},}
\end{array}\right.\label{eq:helmholtz-genericity}
\end{equation}
where $a\in C^{0,\alpha}(\overline{\Omega};\R^{3\times3})$ and $\q\in L^{\infty}(\Omega;\R)$
satisfy
\begin{equation}
\begin{aligned} & \Lambda^{-1}|\xi|^{2}\le a\xi\cdot\xi\le\Lambda|\xi|^{2},\qquad\xi\in\R^{3},\\
 & \Lambda^{-1}\le\q\le\Lambda\quad\text{almost everywhere},
\end{aligned}
\label{eq:ass-ellipticity}
\end{equation}
for some $\Lambda>0$. Let $0<\lambda_{1}<\lambda_{2}\le\dots$ denote
the Dirichlet eigenvalues of the above problem counted according to
their multiplicity. Namely, \eqref{eq:helmholtz-genericity} admits
a unique solution $u_{\omega}^{\g}\in H^{1}(\Omega;\C)$ for all $\g\in C^{1,\alpha}(\overline{\Omega};\R)$
and $\omega\in\C\setminus\Sigma$, where $\Sigma=\{\lambda_{l}:l\in\N^{*}\}$.
Let $\eig_{l}\in H_{0}^{1}(\Omega;\R)$ be the associated eigenfunctions
such that
\begin{equation}
-\div(a\nabla\eig_{l})=\lambda_{l}\q\eig_{l}\quad\text{in \ensuremath{\Omega},}\label{eq:eigen-1}
\end{equation}
subject to the normalization $\left\Vert \eig_{l}\right\Vert _{\ldr}^{2}:=\int_{\Omega}\q\eig_{l}^{2}\, dx=1$
(note that, in view of \eqref{eq:ass-ellipticity}, this norm is equivalent
to the standard $L^{2}$ norm). Classical elliptic theory \cite{evans,gilbarg2001elliptic}
gives that $\{\eig_{l}:l\in\N^{*}\}$ is an orthonormal basis of $\ldr$
and an orthogonal basis of $H_{0}^{1}(\Omega;\R)$ and that $u_{\omega}^{\g},\eig_{l}\in C^{1}(\overline{\Omega};\C)$.
Set $\psi_{l}=(a\nabla\eig_{l})\cdot\nu\in C(\partial\Omega;\R)$,
where $\nu$ denotes the unit outer normal to $\bo$.

We consider multiple frequencies in a fixed interval. More precisely,
let $\A=[K_{min},K_{max}]$ denote the range of admissible frequencies
for some $K_{min}<K_{max}$. From this interval, a finite number of
frequencies will be selected in the following way. For $n\in\N$,
let $K^{(n)}$ denote the uniform sampling of $\A$ such that $\#K^{(n)}=2^{n}+1$,
namely
\begin{equation}
K^{(n)}=\{K_{min}+2^{-n}(i-1)(K_{max}-K_{min}):i=1,\dots,2^{n}+1\}.\label{eq:K^(n)}
\end{equation}
Note that $K^{(n)}\subseteq K^{(n+1)}$.

Our results are valid for a particular class of boundary conditions.
\begin{defn}
\label{def:admissible}We say that $\g\in C^{1,\alpha}(\overline{\Omega};\R)$
is \emph{admissible} if for every $x\in\overline{\Omega}$ there exists
$\lambda\in\Sigma$ such that
\[
\sum_{l:\lambda_{l}=\lambda}(\g,\psi_{l})_{L^{2}(\bo;\R)}\nabla\eig_{l}(x)\neq0.
\]

\end{defn}
We shall show below that, in most situations, the admissibility is
a generic property%
\footnote{We say that a property is generic if it holds in a residual subset,
i.e.\ a countable intersection of dense open sets.%
}. The main result of this paper states that using multiple frequencies
it is possible to satisfy the constraint $\nabla u_{\omega}^{\g}(x)\neq0$
everywhere in $\overline{\Omega}$, provided that $\g$ is admissible. 
\begin{thm}
\label{thm:genericity}Let $\Omega\subseteq\R^{3}$ be a $C^{1,\alpha}$
bounded domain, $a\in C^{0,\alpha}(\overline{\Omega};\R^{3\times3})$
and $\q\in L^{\infty}(\Omega;\R)$ be such that \eqref{eq:ass-ellipticity}
holds true and $\g\in C^{1,\alpha}(\overline{\Omega};\R)$ be admissible.

There exists $n\in\N$ and a finite open cover
\[
\overline{\Omega}=\bigcup_{\omega\in K^{(n)}\setminus\Sigma}\Omega_{\omega},
\]
where $K^{(n)}$ is given by \eqref{eq:K^(n)}, such that for every $\omega\in K^{(n)}\setminus\Sigma$ there holds
\begin{equation}
|\nabla u_{\omega}^{\g}(x)|>0,\qquad x\in\Omega_{\omega}.\label{eq:constraint}
\end{equation}

\end{thm}
In other words, if the boundary value $\g$ is admissible, it is possible to find a finite set of frequencies $\{\omega\}\subseteq\A$ and corresponding subdomains $\Omega_\omega$ covering $\overline{\Omega}$ such that each subdomain does not contain any critical points of $u_\omega^\g$. As mentioned in the Introduction, considering multiple solutions and multiple subdomains is necessary, since solutions to  \eqref{eq:helmholtz-genericity} are oscillatory, and critical points will in general occur. Thus, it is not possible to enforce the above constraint globally with a single solution. This can already be seen in 1D with $a\equiv q\equiv 1$, where the solutions will be of the type $u^\g_\k(x) =c_1 \cos ( \sqrt{\k} \,x + c_2)$ for some $c_1,c_2\in\R$ depending on $\Omega$, $\g$ $\omega$. The critical points are unavoidable, but depend on the frequency: this is the key ingredient of this result.

The following example shows that there exist \emph{occulting }illuminations
$\g$ so that Theorem~\ref{thm:genericity} fails: the assumption on the admissibility
of $\g$ is crucial.
\begin{example}
\label{exa:occulting}For simplicity we consider the one-dimensional
case, but similar phenomena appear in higher dimensions too \cite{alberti2013multiple}.
Take $\Omega=(0,2)$, $a\equiv\q\equiv1$. The solution to \eqref{eq:helmholtz-genericity}
corresponding to the boundary value $\g\equiv1$ is $u_{\omega}^{\g}(x)=\cos(\sqrt{\omega}(x-1))/\cos(\sqrt{\omega})$,
whence
\[
\partial_{x}u_{\k}^{\g}(1)=0,\qquad\omega\in\R\setminus\Sigma.
\]
Therefore, in $x=1$ the constraint \eqref{eq:constraint} cannot
be satisfied for any choice of the frequency. By Theorem~\ref{thm:genericity},
$\g\equiv1$ cannot be admissible. To see this directly, observe that
$\phi_{l}'(1)=0$ for all $l\in2\N+1$ and $\phi_{l}'(0)=\phi_{l}'(2)$
for all $l\in2\N^{*}$, since $\phi_{l}(x)=\sin(l\pi x/2)$. This
implies that $(1,\psi_{l})_{L^{2}(\bo;\R)}\eig_{l}'(1)=0$ for every
$l\in\N^{*}$, namely $\g\equiv1$ is not admissible.
\end{example}
This example shows that the admissibility condition is not tautological:
there exist non admissible boundary values $\g$. However, the chosen
$g$ was in some sense pathological, as the following lemma shows.
\begin{lem}
\label{lem:admissible generic cuboid}If $\Omega=\prod_{i=1}^{3}(0,b_{i})$
for some $b_{i}>0$ and $a\equiv\q\equiv1$ in $\Omega$, then for
a generic $\g\in C^{1,\alpha}(\overline{\Omega};\R)$ and every $x\in\Omega$
there exists $\lambda\in\Sigma$ such that
\[
\sum_{l:\lambda_{l}=\lambda}(\g,\psi_{l})_{L^{2}(\bo;\R)}\partial_{x_{1}}\eig_{l}(x)\neq0.
\]

\end{lem}
In the general case, there is a big class of problems for which the
admissibility is a generic property. We do not know if the admissibility
is always a generic property.
\begin{prop}
\label{prop:admissible is generic}If $\Omega\subseteq\R^{3}$ is
a $C^{1,\alpha}$ bounded domain, $a\in C^{2}(\overline{\Omega};\R^{3\times3})$
and $\q\in C^{1}(\overline{\Omega};\R)$ satisfy \eqref{eq:ass-ellipticity}
and all the eigenvalues $\{\lambda_{l}\}$ are simple, then a generic
$\g\in C^{1,\alpha}(\overline{\Omega};\R)$ is admissible.
\end{prop}
Therefore, if the spectrum is simple, Theorem~\ref{thm:genericity}
holds for a generic $\g\in C^{1,\alpha}(\overline{\Omega};\R)$.
Moreover, the simplicity of the spectrum is a generic property. Indeed,
for general elliptic operator on bounded domains $\Omega$, the eigenvalues
$\lambda_{l}$ may have multiplicity bigger than one (as in Lemma~\ref{lem:admissible generic cuboid}).
However, this is a consequence of special symmetries in $\Omega$
and in the parameters:  for a generic elliptic operator the eigenvalues
are all simple \cite{uhlenbeck-1976,albert-1978}. More precisely,
under the assumptions $a\in C^{2}(\R^{3};\R^{3\times3})$, $\q\in C^{2}(\R^{3};\R)$
and $a=a^{T}$, for a generic $F\in{\rm Diff}(\Omega)$, all the eigenvalues
of \eqref{eq:eigen-1} in $\Omega^{F}$ are simple \cite[Example 6.3]{henry_2005}.
Here, ${\rm Diff}(\Omega)$ denotes the open subset of $C^{2}(\Omega;\R^{3})$
consisting of maps $F$ which are diffeomorphisms to their images
$\Omega^{F}:=F(\Omega)$. This implies the following corollary.
\begin{cor}
\label{cor:genericity}Let $\Omega\subseteq\R^{3}$ be a $C^{2}$
bounded domain and $a\in C^{2}(\R^{3};\R^{3\times3})$ and $\q\in C^{2}(\R^{3};\R)$
be such that $a=a^{T}$ and such that \eqref{eq:ass-ellipticity} hold true.
For a generic $F\in{\rm Diff}(\Omega)$ and a generic $\g\in C^{1,\alpha}(\overline{\Omega^{F}};\R)$
there exist $n\in\N$ and a finite open cover
\[
\overline{\Omega^F}=\bigcup_{\omega\in K^{(n)}\setminus\Sigma}\Omega_{\omega}
\]
such that for every $\omega\in K^{(n)}\setminus\Sigma$ there holds
\[
|\nabla u_{\omega}^{\g}(x)|>0,\qquad x\in\Omega_{\omega},
\]
where $u_{\omega}^{\g}$ is the solution to \eqref{eq:helmholtz-genericity}
in $\Omega^{F}$ and $\Sigma$ is the corresponding spectrum.
\end{cor}
\begin{rem}
By using a simple compactness argument \cite{alberti2013multiple},
it is possible to show that in Theorem~\ref{thm:genericity} and
Corollary~\ref{cor:genericity} we actually have
\[
|\nabla u_{\omega}^{\g}(x)|\ge C,\qquad\omega\in K^{(n)}\setminus\Sigma,\; x\in\Omega_{\omega}
\]
for some $C>0$. However, no quantitative a priori estimates on $n$
and $C$ can be given, since $n\to\infty$ and $C\to0$ as the chosen
$\g$ approaches a non admissible boundary condition. This is in contrast
with the results obtained in \cite{2014-albertigs}, where $n$ and
$C$ are given a priori, depending on $a$ and $\q$ only through
their a priori bounds.
\end{rem}

\begin{rem}
A precise estimate on the number of frequencies needed can be obtained
under the assumption of the analyticity of $a$ and $q$, by using
the result in \cite{alberti-capdeboscq-analytic}. More precisely,
if $a$ and $q$ are analytic, then the conclusions of Theorem~\ref{thm:genericity}
and Corollary~\ref{cor:genericity} hold true for any choice of four
frequencies in an open and dense set of $\A^{4}$.
\end{rem}

\begin{rem}
The results presented in this work state that, under certain spectral
assumptions, it is possible to control the behavior of $\nabla u_{\omega}^{\g}$
by choosing several ``controls'' $\omega\in\A$, provided that $\g$
is not pathological. The conditions on $\g$ are expressed by means
of  Fourier expansions. In these terms, the theory discussed here
is linked with the optimal observability for the heat and wave equations
discussed in \cite{privat-2013,privat-2015,privat2012optimal}, where
the optimality refers to random choices of the (Fourier coefficients
of the) initial data.
\end{rem}
We finally note that all the results presented in this paper are valid
in any dimension, provided that the regularity assumptions on the
coefficients are changed accordingly.

\section{\label{sec:Proof-of-Theorem}Proof of Theorem~\ref{thm:genericity}}

The proof of Theorem~\ref{thm:genericity} is based on the holomorphicity
of $u_{\k}^{\g}$ with respect to $\k$.
\begin{lem}
\label{lem:holomorphic}Let $\Omega\subseteq\R^{3}$ be a $C^{1,\alpha}$
bounded domain and $a\in C^{0,\alpha}(\overline{\Omega};\R^{3\times3})$
and $\q\in L^{\infty}(\Omega;\R)$ be such that \eqref{eq:ass-ellipticity}
holds true. Take $\g\in C^{1,\alpha}(\overline{\Omega};\R)$. Then
the map
\[
\k\in\C\setminus\Sigma\longmapsto u_{\k}^{\g}\in C^{1}(\overline{\Omega};\C)
\]
is holomorphic and for every $m\in\N^{*}$ there holds
\begin{equation}
\left\{ \begin{array}{l}
-\div(a\,\nabla(\partial_{\k}^{m}u_{\k}^{\g}))-\k\,\q\,\partial_{\k}^{m}u_{\k}^{\g}=m\q\partial_{\k}^{m-1}u_{\k}^{\g}\quad\text{\text{in \ensuremath{\Omega},}}\\
\partial_{\k}^{m}u_{\k}^{\g}=0\quad\text{on \ensuremath{\partial\Omega}.}
\end{array}\right.\label{eq:dk uk}
\end{equation}
\end{lem}
\begin{proof}
The holomorphicity of the map $\k\in\C\setminus\Sigma\mapsto u_{\k}^{\g}\in C^{1}(\overline{\Omega};\C)$
is a consequence of the holomorphicity of the resolvent operator and
of classical elliptic regularity theory. The details are given in
\cite[Proposition 3.5]{alberti2013multiple}, where it is shown that
for $\omega,\omega'\in\C\setminus\Sigma$ sufficiently close there
holds
\[
u_{\k'}^{\g}=\sum_{m=0}^{\infty}(BC_{\k'-\k})^{m}u_{\k}^{\g},
\]
where $B\colon L^{\infty}(\Omega;\C)\to C^{1}(\overline{\Omega};\C)$
is defined by
\[
\left\{ \begin{array}{l}
-\div(a\,\nabla Bf)-\k\q Bf=f\qquad\text{in \ensuremath{\Omega},}\\
Bf=0\qquad\text{on \ensuremath{\partial\Omega},}
\end{array}\right.
\]
and $C_{h}w=hM_{\q}w=h\q w$. Therefore we have for all $\omega'$
sufficiently close to $\omega$
\[
u_{\k'}^{\g}=\sum_{m=0}^{\infty}(BC_{\k'-\k})^{m}u_{\k}^{\g}=\sum_{m=0}^{\infty}(BM_{\q})^{m}u_{\k}^{\g}(\k'-\k)^{m},
\]
whence $\partial_{\k}^{m}u_{\k}^{\g}=m!(BM_{\q})^{m}u_{\k}^{\g}=m(BM_{\q})\partial_{\k}^{m-1}u_{\k}^{\g}$.
Equivalently, we have shown that
\[
\partial_{\k}^{m}u_{\k}^{\g}=B(m\q\partial_{\k}^{m-1}u_{\k}^{\g}),
\]
 as desired.
\end{proof}
We need the following result on the asymptotic distribution of the
eigenvalues. The result is classical, and in the case $a\equiv\q\equiv1$
it is known as Weyl's lemma. We provide a proof for completeness.
\begin{lem}
\label{lem:weyl's lemma}Let $\Omega\subseteq\R^{3}$ be a $C^{1,\alpha}$
bounded domain and $a\in C^{0,\alpha}(\overline{\Omega};\R^{3\times3})$
and $\q\in L^{\infty}(\Omega;\R)$ be such that \eqref{eq:ass-ellipticity}
holds true. There exist $C_{1},C_{2}>0$ depending on $\Omega$ and
$\maxa$ such that
\[
C_{1}l^{\frac{2}{3}}\le\lambda_{l}\le C_{2}l^{\frac{2}{3}},\qquad l\in\N^{*}.
\]
\end{lem}
\begin{proof}
Let $\mathfrak{F}_{l}$ denote the set of all $l$-dimensional subspaces
of $\Vr$. In view of the Courant\textendash{}Fischer\textendash{}Weyl
min-max principle \cite[Exercise 12.4.2]{schmudgen2012} we have $\lambda_{l}=\min_{D\in\mathfrak{F}_{l}}\max_{u\in D\setminus\{0\}}\frac{\int_{\Omega}a\nabla u\cdot\nabla u\, dx}{\int_{\Omega}\q u^{2}\, dx}$
for every $l\in\N^{*}$. Therefore we have
\begin{equation}
\Lambda^{-2}\mu_{l}\le\lambda_{l}\le\Lambda^{2}\mu_{l},\qquad l\in\N^{*},\label{eq:lambda_mu}
\end{equation}
where $\mu_{l}=\min_{D\in\mathfrak{F}_{l}}\max_{u\in D\setminus\{0\}}\frac{\int_{\Omega}\nabla u\cdot\nabla u\, dx}{\int_{\Omega}u^{2}\, dx}.$
By the min-max principle, $\mu_{l}$ are the eigenvalues of the Laplace
operator on $\Omega$, and so they satisfy $c_{1}l^{\frac{2}{d}}\le\mu_{l}\le c_{2}l^{\frac{2}{d}}$
for some $c_{1},c_{2}>0$ depending on $\Omega$ (see \cite[Theorem 12.14]{schmudgen2012}
or \cite[Chapter 5, Lemma 3.1]{kavian1993}). Combining this inequality
with \eqref{eq:lambda_mu} yields the result.
 \end{proof}
The following result is the main step of the proof of Theorem~\ref{thm:genericity}.
\begin{prop}
\label{prop: inclusion}Let $\Omega\subseteq\R^{3}$ be a $C^{1,\alpha}$
bounded domain and $a\in C^{0,\alpha}(\overline{\Omega};\R^{3\times3})$
and $\q\in L^{\infty}(\Omega;\R)$ be such that \eqref{eq:ass-ellipticity}
holds true and $\g\in C^{1,\alpha}(\overline{\Omega};\R)$ be admissible.
For every $x\in\overline{\Omega}$ there exists $\omega\in\C\setminus\Sigma$
such that $\nabla u_{\k}^{\g}(x)\neq0$.\end{prop}
\begin{proof}
We need the following estimate, which follows from classical elliptic
regularity theory \cite[Section 8.11]{gilbarg2001elliptic}. The eigenfunctions
$\eig_{l}$ $\in C^{1}(\overline{\Omega};\R)$ and
\begin{equation}
\left\Vert \eig_{l}\right\Vert _{C^{1}(\overline{\Omega};\R)}\le c\lambda_{l}^{P},\qquad l\in\N^{*}\label{eq:phi_i in C^1}
\end{equation}
for some $P,c>0$ depending on $\Omega$ and $\maxa$ only.

Take $v\in\Honer$ such that $\Delta v=0$ and $v=\g$ on $\bo$.
Standard integrations by parts yield $-(\g,\psi_{l})_{L^{2}(\partial\Omega;\R)}=\lambda_{l}(\eig_{l},v)_{\ldr}$,
whence
\begin{equation}
\bigl|(\g,\psi_{l})_{L^{2}(\partial\Omega;\R)}\bigr|\le\lambda_{l}\left\Vert \eig_{l}\right\Vert _{\ldr}\left\Vert v\right\Vert _{\ldr}\le C\lambda_{l}\left\Vert \g\right\Vert _{H^{1/2}(\bo;\R)}\label{eq:psi l g}
\end{equation}
for some $C>0$ depending only on $\Omega$ and $\maxa$.

By contradiction, assume that there exists $x\in\overline{\Omega}$
such that $\nabla u_{\k}^{\g}(x)=0$ for every $\k\in\C\setminus\Sigma$.
Therefore for any $m\in\N$ there holds 
\begin{equation}
\nabla(\partial_{\k}^{m}u_{\k}^{\g})(x)=0,\qquad\k\in\C\setminus\Sigma.\label{eq:2}
\end{equation}
By Definition~\ref{def:admissible}, there exists $\lambda\in\Sigma$
such that 
\begin{equation}
\sum_{l:\lambda_{l}=\lambda}(\g,\psi_{l})_{L^{2}(\bo;\R)}\nabla\eig_{l}(x)\neq0.\label{eq:1}
\end{equation}

From \eqref{eq:helmholtz-genericity} and \eqref{eq:dk uk} we have
for every $l\in\N^{*}$ and $\k\in(0,\lambda)\setminus\Sigma$ 
\begin{equation}
(\partial_{\k}^{m}u_{\k}^{\g},\eig_{l})_{\ldr}=\frac{m!}{(\lambda_{l}-\k)^{m}}(u_{\k}^{\g},\eig_{l})_{\ldr}=\frac{(-1)^{m}m!}{(\lambda_{l}-\k)^{m+1}}(\g,\psi_{l})_{L^{2}(\partial\Omega;\R)}.\label{eq:3}
\end{equation}
Therefore, for every  $\k\in(0,\lambda)\setminus\Sigma$ and $l>\max\{l'\in\N^{*}:\lambda_{l'}=\lambda\}$
there holds
\begin{equation}
\begin{split}\left\Vert (\partial_{\k}^{m}u_{\k}^{\g},\eig_{l})_{\ldr}\eig_{l}\right\Vert _{\Coner} & \le c(\Omega,\maxa)\left|(\partial_{\k}^{m}u_{\k}^{\g},\eig_{l})_{\ldr}\right|\lambda_{l}^{P}\\
 & \le c(\Omega,\maxa,m)\frac{\bigl|(\g,\psi_{l})_{L^{2}(\partial\Omega;\R)}\bigr|\lambda_{l}^{P}}{\left|\lambda_{l}-\k\right|{}^{m+1}}\\
 & \le c(\Omega,\maxa,m,\g)\frac{\lambda_{l}^{P+1}}{\left|\lambda_{l}-\k\right|{}^{m+1}}\\
 & \le c(\Omega,\maxa,m,\g)\frac{\lambda_{l}^{P+1}}{(\lambda_{l}-\lambda){}^{m+1}},
\end{split}
\label{eq:bound}
\end{equation}
where the first inequality follows from \eqref{eq:phi_i in C^1} and
the third inequality from \eqref{eq:psi l g}. In view of Lemma~\ref{lem:weyl's lemma},
it is possible to choose $m$ large enough so that we have the convergence
of the series
\[
\sum_{l:\lambda_{l}>\lambda}^{\infty}\frac{\lambda_{l}^{P+1}}{(\lambda_{l}-\lambda){}^{m+1}}<\infty.
\]
Therefore, by \eqref{eq:bound} we obtain that the Fourier series
$\partial_{\k}^{m}u_{\k}^{\g}=\sum_{l}(\partial_{\k}^{m}u_{\k}^{\g},\eig_{l})_{\ldr}\eig_{l}$
converges in $C^{1}(\overline{\Omega};\R)$. Hence, we have
\[
\nabla(\partial_{\k}^{m}u_{\k}^{\g})(x)=\sum_{l\in\N^{*}}(\partial_{\k}^{m}u_{\k}^{\g},\eig_{l})_{\ldr}\nabla\eig_{l}(x),\qquad\k\in(0,\lambda)\setminus\Sigma.
\]
As a result, combining \eqref{eq:2} and \eqref{eq:3} we have $0=\sum_{l}\frac{1}{(\lambda_{l}-\k)^{m+1}}(\g,\psi_{l})_{L^{2}(\partial\Omega;\R)}\nabla\eig_{l}(x)$,
whence
\[
\frac{-1}{(\lambda-\k)^{m+1}}\sum_{l:\lambda_{l}=\lambda}(\g,\psi_{l})_{L^{2}(\partial\Omega;\R)}\nabla\eig_{l}(x)=\sum_{l:\lambda_{l}\neq\lambda}\frac{1}{(\lambda_{l}-\k)^{m+1}}(\g,\psi_{l})_{L^{2}(\partial\Omega;\R)}\nabla\eig_{l}(x),
\]
for all $\k\in(0,\lambda)\setminus\Sigma$. The series on the right
hand side of this equality converges for $\omega=\lambda_{}$. Thus,
letting $\k\to\lambda$ we obtain a contradiction with \eqref{eq:1}.
\end{proof}
As a consequence of this result and of the holomorphicity of $\omega\mapsto u_{\omega}^{\g}$,
for every $x\in\overline{\Omega}$ the set $\{\k\in\A\setminus\Sigma:|\nabla u_{\omega}^{\g}|(x)=0\}$
is discrete. In other words, the critical set $\{x\in\overline{\Omega}:|\nabla u_{\omega}^{\g}|(x)=0\}$
``moves'' when the frequency changes. This is the main idea for
proving Theorem~\ref{thm:genericity}.

\begin{proof}[Proof of Theorem~\ref{thm:genericity}]
By construction of $K^{(m)}$, it is possible to choose $\k_{m}\in K^{(m)}\setminus\Sigma$
and $\k\in\A\setminus\Sigma$ such that $\k_{m}\to\k$ and $\k_{m}\neq\k$
for all $m\in\N$. Indeed, it is enough to choose $\k\in\A\setminus(\cup_{m}K^{(m)}\cup\Sigma)$
and use the fact that $\overline{\cup_{m}(K^{(m)}\setminus\Sigma)}=\A$.
By Lemma~\ref{lem:holomorphic}, for every $x\in\overline{\Omega}$
the map
\[
\zeta_{x}\colon\C\setminus\Sigma\to\C,\qquad\omega\mapsto\nabla u_{\k}^{\g}(x)\cdot\overline{\nabla u_{\overline{\k}}^{\g}}(x)
\]
is holomorphic. Since $\g$ is real-valued, there holds
\begin{equation}
\zeta_{x}(\omega)=|\nabla u_{\k}^{\g}(x)|^{2},\qquad\omega\in\R\setminus\Sigma.\label{eq:sigma real}
\end{equation}
By Proposition~\ref{prop: inclusion}, there exists $\omega^{*}\in\C\setminus\Sigma$
such that $\nabla u_{\k^{*}}^{\g}(x)\neq0$. As a consequence, $\partial_{x_{i}}u_{\k^{*}}^{\g}(x)\neq0$
for some $i$. Since $\omega\mapsto\partial_{x_{i}}u_{\k}^{\g}(x)$
is holomorphic by Lemma~\ref{lem:holomorphic}, we can assume without
loss of generality that $\omega^{*}\in\R$. Therefore, by \eqref{eq:sigma real}
we obtain $\zeta_{x}(\omega^{*})\neq0.$ We have shown that $\zeta_{x}$
is holomorphic and non-zero. Hence, by the analytic continuation theorem
there exists $m_{x}\in\N$ such that $\zeta_{x}(\omega_{m_{x}})\neq0$.
Thus, since $\omega_{m}$ is real for every $m$ by \eqref{eq:sigma real}
we obtain $|\nabla u_{\k_{m_{x}}}^{\g}|(x)\neq0$. As a result, since
$y\mapsto|\nabla u_{\k_{m_{x}}}^{\g}|(y)$ is continuous in $\overline{\Omega}$
there exists $r_{x}>0$ such that 
\begin{equation}
|\nabla u_{\k_{m_{x}}}^{\g}|(y)\neq0,\qquad y\in B(x,r_{x})\cap\overline{\Omega}.\label{eq:no name}
\end{equation}
Since $\overline{\Omega}=\cup_{x\in\overline{\Omega}}(B(x,r_{x})\cap\overline{\Omega})$
and $\overline{\Omega}$ is compact, there exist $x_{1},\dots,x_{N}\in\overline{\Omega}$
such that
\[
\overline{\Omega}=\bigcup_{i=1}^{N}\left(B(x_{i},r_{x_{i}})\cap\overline{\Omega}\right).
\]
Therefore, choosing $n=\max\{m_{x_{1}},\dots,m_{x_{N}}\}$ so that
$K^{(m_{x_{i}})}\subseteq K^{(n)}$ and defining for $\omega\in K^{(n)}\setminus\Sigma$
\[
\Omega_{\omega}=\begin{cases}
B(x_{i},r_{x_{i}})\cap\overline{\Omega} & \text{if \ensuremath{\omega=\omega_{m_{x_{i}}}}for some \ensuremath{i=1,\dots,N}, or}\\
\emptyset & \text{otherwise,}
\end{cases}
\]
the result follows by \eqref{eq:no name}.
\end{proof}

\section{\label{sec:The-genericity-of}The Genericity of the Admissibility
Property}

In this section we prove Proposition~\ref{prop:admissible is generic} and Lemma~\ref{lem:admissible generic cuboid}, that will be a
consequence of Lemmata~\ref{lem:case 1} and \ref{lem:case 2}.

The following lemma states that for every $x\in\overline{\Omega}$
there exists at least one eigenfunction with non-zero gradient in
$x$. The proof is based on the pointwise convergence of the Fourier
series of a compactly supported smooth function.
\begin{lem}
\label{lem:i star}Let $\Omega\subseteq\R^{3}$ be a $C^{1,\alpha}$
bounded domain and $a\in C^{2}(\overline{\Omega};\R^{3\times3})$
and $\q\in C^{1}(\overline{\Omega};\R)$ be such that \eqref{eq:ass-ellipticity}
holds true. For every $x\in\overline{\Omega}$ there exists $l\in\N^*$
such that $\nabla\eig_{l}(x)\neq0$.\end{lem}
\begin{proof}
If $x\in\bo$, then by the Hopf lemma we have $\frac{\partial\phi_{1}}{\partial\nu}(x)\neq0$,
namely $\nabla\eig_{1}(x)\neq0$.

Assume now $x\in\Omega$. It is enough to show that if $f\in C_{0}^{\infty}(\Omega;\R)$
($C^{\infty}$ with compact support contained in $\Omega)$ then $f=\sum_{l}(f,\eig_{l})_{\ldr}\eig_{l}$
converges in $C^{1}(\overline{\Omega};\R)$. Then it suffices to take
$f\in C_{0}^{\infty}(\Omega;\R)$ such that $\nabla f(x)\neq0$.

Define $f_{n}=\sum_{l\ge n}(f,\eig_{l})_{\ldr}\eig_{l}\in\Vr$. We
shall show that $f_{n}\to0$ in $C^{1}(\overline{\Omega};\R)$. Let
$h\in\Vr$ be defined by $-\div(a\nabla f)=\q h$. A direct calculation
shows that
\[
-\div(a\nabla f_{n})=\q\, h_{n}\qquad\text{in \ensuremath{\Omega},}
\]
where $h_{n}=\sum_{l\ge n}(h,\eig_{l})_{\ldr}\eig_{l}$. Since $h\in\Vr$ and $\{\eig_{l}:l\in\N^{*}\}$
is an orthogonal basis for $\Vr$, we have the convergence $h_{n}\to0$
in $\Vr$. By the regularity assumption on $a$ and $\q$, this implies
$f_{n}\to0$ in $H^{3}(\Omega;\R)$. Finally, the result follows by
the Sobolev embedding theorem.
\end{proof}
The generic set of boundary conditions we are going to consider is
the intersection of the sets given by 
\[
O_{l}(\Omega)=\{\g\in C^{1,\alpha}(\overline{\Omega};\R):(\g,\psi_{l})_{L^{2}(\partial\Omega;\R)}\neq0\},\qquad l\in\N^{*}.
\]
We need the following elementary property.
\begin{lem}
\label{lem:O open dense}Let $\Omega\subseteq\R^{3}$ be a Lipschitz
bounded domain and $a\in C^{1}(\overline{\Omega};\R^{3\times3})$
and $\q\in L^{\infty}(\Omega;\R)$ be such that \eqref{eq:ass-ellipticity}
holds true. Then the set $O_{l}(\Omega)$ is open and dense in $C^{1,\alpha}(\overline{\Omega};\R)$
for every $l\in\N^{*}$.\end{lem}
\begin{proof}
Since the map $\g\mapsto(\g,\psi_{l})_{L^{2}(\partial\Omega;\R)}$
is continuous, the set $O_{l}(\Omega)$ is open. We now show that
it is dense. Take $\g\in C^{1,\alpha}(\overline{\Omega};\R)\setminus O_{l}(\Omega)$.
By the unique continuation for the Cauchy problem for elliptic equations
\cite[Theorem 1.7]{alessandrini_2009}, we have that $\psi_{l}\not\equiv0$.
Therefore, there exists $h\in C^{1,\alpha}(\overline{\Omega};\R)$
such that $(h,\psi_{l})_{L^{2}(\partial\Omega;\R)}\neq0$. As a consequence,
$\g_{n}=\g+h/n\in O_{l}(\Omega)$ and $\g_{n}\to\g$ in $C^{1,\alpha}(\overline{\Omega};\R)$.
\end{proof}
We now study the genericity of the admissibility condition in the
assumptions of Proposition~\ref{prop:admissible is generic}.
\begin{lem}
\label{lem:case 1}Let $\Omega\subseteq\R^{3}$ be a $C^{1,\alpha}$
bounded domain and $a\in C^{2}(\overline{\Omega};\R^{3\times3})$
and $\q\in C^{1}(\overline{\Omega};\R)$ be such that \eqref{eq:ass-ellipticity}
holds true. Assume that all the eigenvalues $\{\lambda_{l}\}$ are
simple and take $\g\in C^{1,\alpha}(\overline{\Omega};\R)$. If $\g\in\cap_{l\in\N^{*}}O_{l}(\Omega)$
then $g$ is admissible. In particular, by Lemma~\ref{lem:O open dense},
a generic $\g\in C^{1,\alpha}(\overline{\Omega};\R)$ is admissible.\end{lem}
\begin{proof}
Assume that $\g\in\cap_{l\in\N^{*}}O_{l}(\Omega)$ and fix $x\in\overline{\Omega}$.
Since the spectrum is simple, we need to show that there exists $l\in\N^{*}$
such that $(g,\psi_{l})_{L^{2}(\bo;\R)}\nabla\eig_{l}(x)\neq0$. By
Lemma~\ref{lem:i star}, there exists $l\in\N^{*}$ such that $\nabla\eig_{l}(x)\neq0$.
Moreover, since $g\in O_{l}(\Omega)$, there holds $(g,\psi_{l})_{L^{2}(\bo;\R)}\neq0$,
whence $(g,\psi_{l})_{L^{2}(\bo;\R)}\nabla\eig_{l}(x)\neq0$, as desired.
\end{proof}
We now study the genericity of the admissibility condition in the
case discussed in Lemma~\ref{lem:admissible generic cuboid}.
\begin{lem}
\label{lem:case 2}Let $b_{1},b_{2},b_{3}>0$, set $\Omega=\prod_{i=1}^{3}(0,b_{i})$,
$a\equiv\q\equiv1$ and take $\g\in C^{1,\alpha}(\overline{\Omega};\R)$.
If $\g\in\cap_{l\in\N^{*}}O_{l}(\Omega)$ then for every $x\in\Omega$
there exists $\lambda\in\Sigma$ such that
\[
\sum_{l:\lambda_{l}=\lambda}(\g,\psi_{l})_{L^{2}(\bo;\R)}\nabla\eig_{l}(x)\neq0.
\]
In particular, by Lemma~\ref{lem:O open dense}, the conclusion of
Lemma~\ref{lem:admissible generic cuboid} holds true for a generic
$\g\in C^{1,\alpha}(\overline{\Omega};\R)$.\end{lem}
\begin{proof}
The eigenvalues and the eigenfunctions of the Laplacian in a rectangle
are well known and are given by
\[
\lambda_{l}=\sum_{i=1}^{3}\left(\frac{l_{i}\pi}{b_{i}}\right)^{2},\quad\eig_{l}(x)=c\prod_{i=1}^{3}\sin(\frac{l_{i}\pi}{b_{i}}x_{i}),\qquad l\in(\N^{*})^{3},
\]
where $c=b_{1}b_{2}b_{3}/8$ is a normalization factor. Assume that
$g\in\cap_{l\in(\N^{*})^{3}}O_{l}(\Omega)$; in particular, we have
\[
(\g,\psi_{(1,1,1)})_{L^{2}(\partial\Omega;\R)}\neq0\quad\text{and}\quad(\g,\psi_{(2,1,1)})_{L^{2}(\partial\Omega;\R)}\neq0.
\]
Fix now $x\in\Omega$. We need to show that there exists $\lambda\in\Sigma$
such that
\[
\sum_{l:\lambda_{l}=\lambda}(\g,\psi_{l})_{L^{2}(\bo;\R)}\nabla\eig_{l}(x)\neq0.
\]

If $x_{1}\neq b_{1}/2$, by choosing $\lambda=\lambda_{(1,1,1)}$
we have
\[
\begin{split}\sum_{l:\lambda_{l}=\lambda}(\g,\psi_{l})_{L^{2}(\bo;\R)}\partial_{x_{1}}\eig_{l}(x) & =(\g,\psi_{(1,1,1)})_{L^{2}(\bo;\R)}\partial_{x_{1}}\eig_{(1,1,1)}(x)\\
 & =(\g,\psi_{(1,1,1)})_{L^{2}(\bo;\R)}c\frac{\pi}{b_{1}}\cos(\frac{\pi}{b_{1}}x_{1})\prod_{i=2}^{3}\sin(\frac{\pi}{b_{i}}x_{i})\\
 & \neq0,
\end{split}
\]
as desired.

If $x_{1}=b_{1}/2$, by choosing $\lambda=\lambda_{(2,1,1)}$ we have
\[
\begin{split}\sum_{l:\lambda_{l}=\lambda}(\g,\psi_{l})_{L^{2}(\bo;\R)}\partial_{x_{1}}\eig_{l}(x) & =\sum_{l:\lambda_{l}=\lambda}(\g,\psi_{l})_{L^{2}(\bo;\R)}c\frac{l_{1}\pi}{b_{1}}\cos(\frac{l_{1}\pi}{2})\prod_{i=2}^{3}\sin(\frac{l_{i}\pi}{b_{i}}x_{i})\\
 & =-(\g,\psi_{(2,1,\dots,1)})_{L^{2}(\bo;\R)}c\frac{2\pi}{b_{1}}\prod_{i=2}^{3}\sin(\frac{\pi}{b_{i}}x_{i})\\
 & \neq0,
\end{split}
\]
where the second equality comes from the fact that $\cos(\frac{l_{1}\pi}{2})=0$
if $l_{1}=1$ (note that $\lambda_{l}=\lambda_{(2,1,1)}$ implies
$l_{1}\in\{1,2\}$). This concludes the proof.
\end{proof}
\begin{rem}
In fact, we have proven that it is sufficient that $\g$ belongs to
$O_{(1,1,1)}(\Omega)\cap O_{(2,1,1)}(\Omega)$ in order to be admissible.
\end{rem}

\section{\label{sec:Towards-the-study}Towards the Study of the Jacobian for Systems of Equations}

The study of the absence of critical points was strongly related with the more
general problem of the Jacobian. As mentioned in the Introduction,
the proof of the Rad\'o-Kneser-Choquet theorem in two dimensions is
based on the absence of critical points for solutions with certain
boundary values. Let us summarize the argument for the case of the
conductivity equation
\[
\left\{ \begin{array}{l}
-\div(a\nabla u^{i})=0\quad\text{in }\Omega,\\
u^{i}=g_{i}\quad\text{on }\partial\Omega,
\end{array}\right.
\qquad i=1,2.
\]
For simplicity, assume that $\Omega$ is convex and choose boundary
values $\g_{i}=x_{i}$ for $i=1,2$. If $\det(\nabla u^{1},\nabla u^{2})(x)=0$
for some $x\in\Omega$, then $\alpha\nabla u^{1}(x)+\beta\nabla u^{2}(x)=0$
for some $\alpha,\beta\in\R$. Setting $u=\alpha u^{1}+\beta u^{2}$,
we have $\nabla u(x)=0$ and
\[
\left\{ \begin{array}{l}
-\div(a\nabla u)=0\quad\text{in }\Omega,\\
u=\alpha x_{1}+\beta x_{2}\quad\text{on }\partial\Omega.
\end{array}\right.
\]
However, $u$ cannot have any critical points in $\Omega$, since
$\alpha x_{1}+\beta x_{2}$ has only one minimum and one maximum on
$\bo$ \cite{alessandrinimagnanini1994}.

Thus, proving the absence of points with vanishing Jacobian boils
down to proving an a priori weaker property on the absence of critical
points for solutions with boundary values satisfying certain conditions.
As a consequence, studying the absence of critical points represents
a natural first step for the study of the Jacobian. Unfortunately,
deducing directly the absence of points with a vanishing Jacobian
from Theorem~\ref{thm:genericity} following the argument outlined
above does not appear to be possible  because of the structure of the admissibility condition.

On the other hand, at the current state we are unable to consider
the Jacobian instead of the constraint $\nabla u_{\omega}^{\g}$ in
Theorem~\ref{thm:genericity}. This can be seen from the proof of
Proposition~\ref{prop: inclusion}. Indeed, the trick of choosing
a higher order derivative with respect to $\omega$ in order to make
the Fourier series of $\nabla(\partial_{\k}^{m}u_{\k}^{\g})$ pointwise
convergent does not work in the case of the Jacobian. In general,
the problem of the pointwise convergence of the Fourier series associated
to an elliptic operator in a bounded domain is a very delicate issue
\cite{alimov-i,alimov-II}.

One may think that the apparent difficulties in the study of the Jacobian
may be related to the fact that, in general, the Jacobian does not
fulfill the property of unique continuation \cite{jin-kazdan-1991,alessandrini-nesi-2015},
in contrast to the full gradient \cite{garofalo-lin-1986}. However,
this is not the obstacle for this approach. Indeed, as we shall briefly
see below, the analogous of all the results stated in Section~\ref{sec:Main-result}
hold true for the stronger constraint
\begin{equation}
|\partial_{x_{1}}u_{\omega}^{\g}(x)|>0,\label{eq:simple contrs-1}
\end{equation}
for which the property of unique continuation fails \cite{jin-kazdan-1991}.

Let us for simplicity consider the above constraint only in a subset
$\Omega'\Subset\Omega$. We say that $\g\in C^{1,\alpha}(\overline{\Omega};\R)$
is \emph{strongly admissible} if for every $x\in\Omega$ there exists
$\lambda\in\Sigma$ such that
\[
\sum_{l:\lambda_{l}=\lambda}(\g,\psi_{l})_{L^{2}(\bo;\R)}\partial_{x_{1}}\eig_{l}(x)\neq0.
\]
By Lemma~\ref{lem:admissible generic cuboid}, the strong admissibility
is a generic property in a cuboid with constant coefficient. The proof
of Proposition~\ref{prop:admissible is generic} can be easily extended
to this more general case. Thus, the strong admissibility property
is a generic property if the spectrum is simple. In particular, by
arguing as in Corollary~\ref{cor:genericity}, this is true for a
generic $C^{2}$ bounded domain $\Omega$.

A careful inspection of the proof of Theorem~\ref{thm:genericity}
shows that it can be generalized to the constraint given by \eqref{eq:simple contrs-1},
under the assumption of strong admissibility.
\begin{prop}
Let $\Omega\subseteq\R^{3}$ be a $C^{1,\alpha}$ bounded domain,
$a\in C^{0,\alpha}(\overline{\Omega};\R^{3\times3})$ and $\q\in L^{\infty}(\Omega;\R)$
be such that \eqref{eq:ass-ellipticity} holds true and $\g\in C^{1,\alpha}(\overline{\Omega};\R)$
be strongly admissible. Take $\Omega'\Subset\Omega$. There exists
$n\in\N$ and a finite open cover
\[
\overline{\Omega'}=\bigcup_{\omega\in K^{(n)}\setminus\Sigma}\Omega_{\omega}
\]
such that for every $\omega\in K^{(n)}\setminus\Sigma$ there holds
\[
|\partial_{x_{1}}u_{\omega}^{\g}(x)|>0,\qquad x\in\Omega_{\omega}.
\]

\end{prop}
This stronger result, together with the method discussed in this paper,
may provide suitable tools for the study of the Jacobian for problem
\eqref{eq:helmholtz-genericity} or other frequency-dependent PDE.

\bibliographystyle{abbrv}
\bibliography{2015-alberti}

\end{document}